\documentclass[oneside, 11pt]{amsart}
\usepackage[utf8]{inputenc}

\usepackage{amsfonts}
\usepackage{amsmath}
\usepackage{amssymb}
\usepackage{amsthm}
\usepackage[a4paper]{geometry}
\usepackage{fullpage}
\usepackage{mathrsfs} 
\usepackage[dvipsnames]{xcolor}
\usepackage{dsfont}
\usepackage{mathtools,diagbox}
\usepackage{mathtools}
\mathtoolsset{showonlyrefs}
\usepackage{caption}
\usepackage[normalem]{ulem}

\usepackage[colorlinks=true, linkcolor=blue, citecolor=blue]{hyperref}

\usepackage{appendix}
\usepackage{enumerate}
\usepackage{tikz-cd} 

\usepackage{orcidlink}

\theoremstyle{plain}
\newtheorem{theorem}{Theorem}
\newtheorem{proposition}[theorem]{Proposition}
\newtheorem{lemma}[theorem]{Lemma}
\newtheorem{corollary}[theorem]{Corollary}
\newtheorem{definition}[theorem]{Definition}
\newtheorem{remark}[theorem]{Remark}

\theoremstyle{definition}

\numberwithin{theorem}{section}
\numberwithin{equation}{section} 

\newcommand{\vertiii}[1]{{\left\vert\kern-0.25ex\left\vert\kern-0.25ex\left\vert #1 
    \right\vert\kern-0.25ex\right\vert\kern-0.25ex\right\vert}}

\newcommand{\Z}{\mathbb{Z}}

\newcommand{\R}{\mathbb{R}}
\newcommand{\N}{\mathbb{N}}

\newcommand{\cN}{\mathcal{N}}

\newcommand{\cR}{\mathcal{R}}

\renewcommand{\P}{\mathbb P}
\newcommand{\eps}{\varepsilon}
\newcommand{\G}{\mathcal C}

\newcommand{\be}{\begin{equation}}
\newcommand{\ee}{\end{equation}}
\newcommand{\ba}{\begin{align}}
\newcommand{\ea}{\end{align}}

\newcommand\numberthis{\addtocounter{equation}{1}\tag{\theequation}}

\title[Local criteria for global comparisons]{Local criteria for global connectivity comparisons: \\ beyond stochastic domination}
\author{Johannes B\"aumler\orcidlink{0000-0002-3823-9899}}
\address{Department of Mathematics, University of California, Los Angeles}
\email{jbaeumler@math.ucla.edu}
\author{Benedikt Jahnel\orcidlink{0000-0002-4212-0065}}
\address{Institut für Mathematische Stochastik, Technische Universit\"at Braunschweig \& Weierstrass Institute for Applied Analysis and Stochastics, Berlin}
\email{jahnel@wias-berlin.de}
\author{Jonas K\"oppl\orcidlink{0000-0001-9188-1883}}
\address{Weierstrass Institute for Applied Analysis and Stochastics, Berlin}
\email{koeppl@wias-berlin.de}
\author{Bas Lodewijks\orcidlink{0000-0001-5624-2410}}
\address{School of Mathematical and Physical Sciences, University of Sheffield}
\email{bas.lodewijks@sheffield.ac.uk}
\author{Lily Reeves\orcidlink{0000-0002-4169-9917}}
\address{Department of Mathematics, California Institute of Technology}
\email{lreeves@caltech.edu}
\author{András Tóbiás\orcidlink{0000-0003-4881-2553}}
\address{Budapest University of Technology and Economics and HUN-REN Alfréd Rényi Institute of Mathematics}
\email{tobias@cs.bme.hu}

\date{\today}
\begin{document}

\begin{abstract} 
We introduce a site-wise domination criterion for local percolation models, which enables the comparison of one-arm probabilities even in the absence of stochastic domination. The method relies on a local-to-global principle: if, at each site, one model is more likely than the other to connect to a subset of its neighbors, for all nontrivial such subsets, then this advantage propagates to connectivity events at all scales. In this way, we obtain a robust alternative to stochastic domination, applicable in all cases where the latter works and in many where it does not. As a main application, we compare classical Bernoulli bond percolation with degree-constrained models, showing that degree constraints enhance percolation, and obtain asymptotically optimal bounds on critical parameters for degree-constrained models.
\end{abstract}

\maketitle

\section{Introduction}
Percolation theory provides a fundamental probabilistic framework for understanding the onset of large-scale connectivity in random media. In its most basic formulation on a lattice, each edge (or site) is independently declared open with probability $p$, and one studies the resulting random subgraph as $p$ varies. The model exhibits a phase transition at a critical threshold $p_c$, below which only finite clusters appear and above which an infinite cluster emerges almost surely.

A recurring theme in percolation theory is to compare the connectivity properties of two different models. The standard tool for this is stochastic domination, which allows one to compare random configurations site by site. However, stochastic domination is often too rigid, and fails in natural situations where a comparison at the level of connectivity is still possible.

In this paper we propose an alternative comparison method, which is strictly more robust: it applies whenever stochastic domination does, and also in cases where it does not. Our main result is a local-to-global principle showing that if one model is more likely than another to connect a vertex to each of its neighbors, then it is also more likely to connect the origin to the boundary of a ball of radius $n$. 

As an application, we analyze degree-constrained percolation models such as $k$-nearest-neighbor graphs, in which the number of open edges incident to each vertex is subject to local restrictions. Such models arise naturally in statistical physics and network theory, and they have been studied in various contexts, see, e.g., \cite{HaggMee96,Pete,BalisterBollobas13,holmes2014degenerate,CHJK24,JKLT} and the references therein. Our criterion enables us to establish new and asymptotically optimal bounds for the critical parameter for various variants of degree-constrained models. 

\medskip
In the following Section~\ref{sec_main}, we present our main result including its proof. In Section~\ref{sec_ex}, we apply the result to models with degree constraints and establish percolation in several previously open parameter regimes for directed and undirected $k$-nearest-neighbor graphs. Additionally, we obtain asymptotically optimal bounds on the critical value for the bidirectional nearest-neighbor model as the dimension tends to infinity. 

\section{Setting and main results}\label{sec_main}
For $u \in \Z^d$ denote by $\mathcal{N}_u = \{v \in \Z^d\colon \|u-v\|_{1} =1\}$ the set of nearest neighbors of $u$. Consider two probability measures $\mathbf{P}$ and $\mathbf{Q}$ on $\mathcal{P}(\mathcal{N}_o)=\left\{ A : A \subseteq \mathcal{N}_o \right\}$. 
Denote by $\mathbb{P}$ and $\mathbb{Q}$ the product probability measure where each vertex $u$ samples a subset $N(u)$ of $\mathcal{N}_u$ according to $\mathbf{P}$ or $\mathbf{Q}$ respectively, independently of all others, such that $\{x-u\colon x\in N(u)\}$ is distributed according to $\mathbf{P}$ or $\mathbf{Q}$ respectively. 
\subsection{The exploration process}\label{sec:exploration-process}
Starting at some fixed vertex, say the origin $o\in \Z^d$, we want to explore the set of vertices that can be reached by only using edges to vertices in $N(o)$ and so on. For this, let $\G_0 = \{o\}$ be the \textit{zeroth generation}. Given the first $(n-1)$ generations, the $n$-th generation is then given by 
\be\label{eq:C}
    \G_n = \bigcup_{u \in \G_{n-1}} N(u) \setminus \Big(\bigcup_{0\le i\le n-1}\G_i\Big). 
\ee
This exploration process may or may not terminate for a finite $n$ in the sense that $\G_n = \emptyset$. In any case, we define the \textit{explored cluster} to be 
\begin{align*}
    \G := \bigcup_{n\ge 0}\G_n. 
\end{align*}
Consider the directed graph with vertex set $\Z^d$ and directed edge set $\left\{ (u,v) : v \in N(u) \right\}$. Then the exploration process described above corresponds to a breadth-first exploration of $\G$, which is the out-component of the origin. In particular, $\mathbb{P}[|\G|=\infty]=\mathbb{P}[o \rightsquigarrow  \infty] =\theta(\mathbf{P})$ is the {\em percolation probability} of the system. Since the models under consideration here are based on an i.i.d.\ random field, the models are ergodic and $\theta(\mathbf{P})>0$ if and only if the model percolates in the sense that $\mathbb{P}[\exists\ \text{an unbounded open directed path starting from the origin}]>0$.  

Define $B_{n} \coloneqq \left\{ x \in \Z^d : \|x\|_{1} \leq n \right\}$ as the closed $\ell_{1}$-ball of radius $n$ around the origin and let $\partial B_{n} \coloneqq B_{n+1} \setminus B_{n}$ be the boundary of that ball.
Now consider the event $\{o \rightsquigarrow \partial B_n\} = \{ \G \cap \partial B_n \neq \emptyset \}$ that there exists a directed path of nearest-neighbor sites in $\mathcal{C}$ from the origin to the boundary of $B_n$. We are interested in comparing the probabilities $\mathbb{Q}[o \rightsquigarrow \partial B_n]$ and $\mathbb{P}[o\rightsquigarrow \partial B_n]$ without appealing to stochastic domination arguments. 
\begin{theorem}[Local-to-global]\label{thm:local-to-global}
    If $\mathbf{P}$, $\mathbf{Q}$ are such that
    \begin{equation}\label{local-comparison}
        \mathbf{P}[N(o) \cap A \neq \emptyset] \leq \mathbf{Q}[N(o) \cap A \neq \emptyset],\qquad\text{for all $\emptyset \subsetneq A \subseteq \mathcal{N}_o$},
    \end{equation}
    then, 
\begin{equation}\label{global-comparison}
        \mathbb{P}[o \rightsquigarrow \partial B_n] \leq \mathbb{Q}[o \rightsquigarrow \partial B_n],\qquad\text{for all $n \in \N$}.
    \end{equation}
\end{theorem}
Statement~\eqref{global-comparison} in particular allows to transfer any decay rates of the probability of the one-arm event from $\mathbb{Q}$ to $\mathbb{P}$ and in particular implies domination of the percolation functions, i.e., 
    \begin{equation}
        \mathbb{P}[o \rightsquigarrow  \infty] \leq \mathbb{Q}[o \rightsquigarrow \infty].
    \end{equation}
Note that our result readily generalizes to vertex-transitive graphs in place of $\Z^d$ and, as will become clear from the proof, it also extends to other global connectivity events beyond the one-arm event.

Further, note that the comparison~\eqref{local-comparison} is always satisfied when $\mathbf{P}$ is stochastically dominated by $\mathbf{Q}$; however, there are also cases where stochastic domination does not hold, yet~\eqref{local-comparison} still applies.
As an example, consider the discrete $k$-nearest-neighbor graph on $\Z^d$, where $\mathbf{Q}$ samples a subset of size $k$ uniformly at random from $\mathcal{N}_o$, and let $\mathbf{P}$ be such that every neighbor of $o$ is included with probability $k/2d$ independently of all others. Neither of the two measures stochastically dominates the other, but the local comparison~\eqref{local-comparison} holds, see Section \ref{sec:k-nn-calculation}. 

\subsection{Discrete interpolation inequality}
The proof of Theorem \ref{thm:local-to-global} relies on a discrete interpolation inequality that is of independent interest.
For the statement and proof, we only need to consider restrictions of $\mathbb{P}$ and $\mathbb{Q}$ to finite balls $B_n \subset \Z^d$. 
We omit these restrictions from the notation to enhance readability. \medskip

For a set $U \subseteq B_n$ the measure $\mathbb{Q}_U$ is defined by independently sampling $N(u) \sim \mathbf{Q}$ for all $u \in U$ and $N(v) \sim \mathbf{P}$ for all $v \in B_n \setminus U$. The family of measures $(\mathbb{Q}_U)_{U \subseteq B_n}$ can be seen as a discrete interpolation between $\mathbb{P}$ and $\mathbb{Q}$ and helps us to lift the local comparison~\eqref{local-comparison} to the global comparison~\eqref{global-comparison} via a piece-by-piece replacement. 

\begin{proposition}(Interpolation)\label{prop:discrete-interpolation}
    Let $n \in \N$, $U \subseteq B_n$ and $a \in B_n$. Then, under Condition~\eqref{local-comparison}, 
    \begin{align*}
        \mathbb{Q}_U[o \rightsquigarrow \partial B_n]\leq \mathbb{Q}_{U \cup \{a\}}[o \rightsquigarrow \partial B_n]. 
    \end{align*}
\end{proposition}

Our main result now follows without effort.
\begin{proof}[Proof of Theorem~\ref{thm:local-to-global}]
Proposition~\ref{prop:discrete-interpolation} directly implies Theorem~\ref{thm:local-to-global} by iteratively adding one more point and noting that $\mathbb Q_\emptyset [o \rightsquigarrow \partial B_n]=\mathbb P[o \rightsquigarrow \partial B_n]$ and $\mathbb Q_{B_n}[o \rightsquigarrow \partial B_n]=\mathbb Q[o \rightsquigarrow \partial B_n]$, since we are only looking at connection events in finite volumes and thus we at most need to add all points in the finite volume under consideration.
\end{proof}

\begin{figure}[ht]
\centering
\includegraphics[width=.5\textwidth]{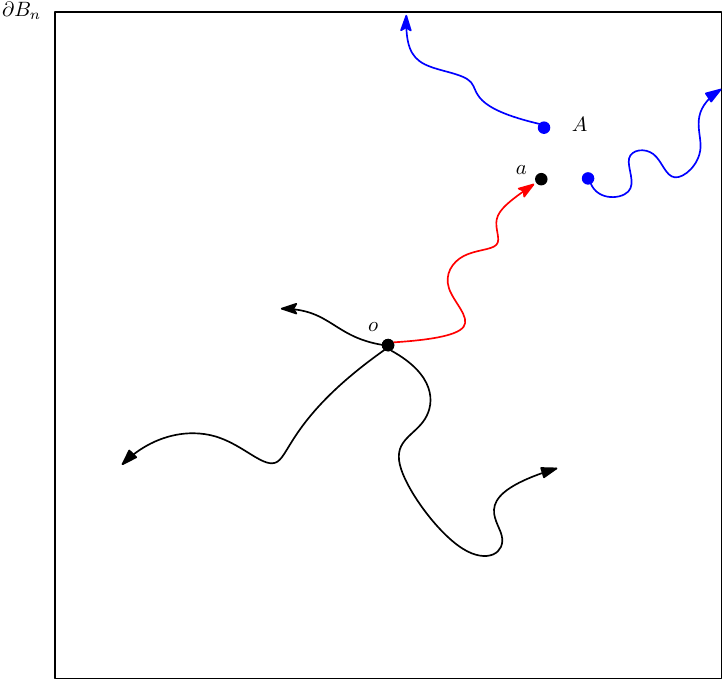}
\caption{\label{fig:local-comparison} Conditioning on all vertices in $B_n \setminus \{a\}$ reduces the question of the origin $o$ being connected to the boundary $\partial B_n$ to a local question: If $a$ is pivotal and $A$ is the set of its neighbors from which one can reach the boundary, then the event $\{o \rightsquigarrow \partial B_n\}$ occurs if and only if $\{N(a)\cap A\neq\emptyset\}$ occurs.}
\end{figure}

\begin{proof}[Proof of Proposition~\ref{prop:discrete-interpolation}]
    Without loss of generality we can assume that $a \notin U$. 
    Define the $\sigma$-algebra $\mathcal{F}_{a^{\rm c}}$ by $\mathcal{F}_{a^{\rm c}} = \sigma \left( N(u) \colon u \in \Z^d \setminus \{a \} \right)$.
    Conditionally on $\mathcal{F}_{a^{\rm c}}$, there are three different cases which can occur. 
    \begin{enumerate}[(1)]
        \item There exists a path from the origin $o$ to the boundary $\partial B_n$ which only uses vertices $u\neq a$. 
        \item Even if $N(a) = \mathcal{N}_a$ were to hold, there can be no path connecting the origin to the boundary. 
        \item There is no path connecting the origin to the boundary that does not use $a$, but there exists at least one path connecting the origin to $a$ and there exists at least one other neighbor of $a$ that is connected to the boundary via an open path. This is the case in which we call $a$ \emph{pivotal} (see Figure~\ref{fig:local-comparison}).
    \end{enumerate}
    Then in the first two cases, there is nothing we have to do, because in those cases the occurrence of the event $\{o \rightsquigarrow \partial B_n\}$ does not depend on $N(a)$. For the last case, note that, conditionally on $\mathcal{F}_{a^{\rm c}}$, the set of neighbors of $a$ that are connected to the boundary is some non-empty set $A \subsetneq \mathcal{N}_a$. So, on this event we can apply \eqref{local-comparison} to obtain 
    \begin{align}\label{eq:discrinterp}
        \mathbb{Q}_U[o \rightsquigarrow \partial B_n \lvert \mathcal{F}_{a^{\rm c}}] 
        &=
        \mathbb{Q}_U[N(a)\cap A\neq\emptyset \lvert \mathcal F_{a^{\rm c}}]= \mathbf P[N(a) \cap A \neq \emptyset]\\
        &\le \mathbf Q[N(a) \cap A \neq \emptyset]
        =
        \mathbb{Q}_{U \cup \{a\}}[N(a)\cap A\neq\emptyset \lvert \mathcal{F}_{a^{\rm c}}].
\end{align}
    By combining this with the first two cases and taking expectations the claim follows. 
\end{proof}

\subsection{Strict comparisons and exponential decay}
In many cases, the comparison~\eqref{local-comparison} is actually strict, in the sense that 
\begin{align}\label{ineq:strict-comparison}
    \mathbf{P}[N(o) \cap A \neq \emptyset] < \mathbf{Q}[N(o) \cap A \neq \emptyset], \quad \text{for all }  \emptyset \subsetneq A \subsetneq \mathcal{N}_0. 
\end{align}
Under the additional assumption that $\mathbb{Q}[o \rightsquigarrow \infty] = 0$, one can then not only deduce that $\mathbb{P}[o \rightsquigarrow \infty] = 0$ via~\eqref{global-comparison} but also derive exponential decay for the probability of the one-arm event.

\begin{theorem}[Strict comparison]\label{thm:strict-comparison}
    Let $\mathbf{P}, \mathbf{Q}$ be such that \eqref{ineq:strict-comparison} holds and additionally assume that $\mathbb{Q}[o \rightsquigarrow \infty] = 0$. Then, there exists a constant $c>0$ such that for all $n\in \N$,
    \begin{align}
        \mathbb{P}[o \rightsquigarrow \partial B_n] \leq {\rm e}^{-cn}. 
    \end{align}
\end{theorem}

The proof of Theorem \ref{thm:strict-comparison} is based on the following general principle that is proved in~\cite{grimmett_dependent_1998}. 

\begin{proposition}[{\cite[Equation (5.3)]{grimmett_dependent_1998}}]\label{prop_exp}
    Let $\mathbf{Q}$ be such that $\mathbb{Q}[o \rightsquigarrow \infty] = 0$. For $p \in [0,1]$ define $\mathbf{Q}_p := p \mathbf{Q} + (1-p) \delta_{\emptyset}$ and denote by $\mathbb{Q}_p$ the law of the associated percolation model.  Then, for any $p \in [0,1)$ there exists a constant $c = c(p,\mathbf{Q})>0$ such that for all $n \in \N$,
    \begin{align*}
        \mathbb{Q}_p[o \rightsquigarrow \partial B_n] \leq {\rm e}^{-cn}. 
    \end{align*}
\end{proposition}

\begin{proof}[Proof of Theorem~\ref{thm:strict-comparison}]
It suffices to note that, due to the strict comparison in~\eqref{ineq:strict-comparison} and continuity, there exists $p \in (0,1)$ such that \eqref{local-comparison} holds for $\mathbf{P}$ and $\mathbf{Q}_p$. By Theorem~\ref{thm:local-to-global} and Proposition~\ref{prop_exp} this implies that
\begin{align*}
    \mathbb{P}[o \rightsquigarrow \partial B_n] \leq \mathbb{Q}_p[o \rightsquigarrow \partial B_n] \leq {\rm e}^{-cn} 
\end{align*}
for some $c> 0$, as claimed. 
\end{proof}


\section{Applications}\label{sec_ex}
In this section, we denote by $\mathbf P_p$ the {\em local i.i.d.\ law} where each vertex in $\Z^d$ draws a directed edge towards each of its nearest neighbors with probability $p\in [0,1]$ independently among said neighbors and also independently of the configurations of the other vertices. We denote by $\mathbb P_p$ the law of the associated directed percolation model.

\subsection{Directed nearest-neighbor models}\label{sec:k-nn-calculation} 
Consider  $\mathbb{Z}^d$ and a parameter $p\in [0,1]$ to which we associate two further parameters
\begin{equation}\label{kepsdef}
    k:=\lfloor 2dp\rfloor\in \{0,\dots,2d\}\quad  \text{ and }\quad  \varepsilon:=2dp-k\in [0,1).
\end{equation}
Given these parameters, we define a directed edge-percolation model, which we call the {\em directed  $2dp$-nearest-neighbor graph} $2dp$-DnG as follows\footnote{Note that, in the earlier work~\cite{JKLT} on degree-constraint nearest-neighbor models, only the integer cases, with $\eps=0$, were treated.}. Each vertex $v\in\mathbb{Z}^d$ chooses, independently and uniformly at random, $k$ of its $2d$ nearest neighbors $\mathcal{N}_v$ and we open the directed edges towards these neighbors. Additionally, with probability $\varepsilon$ and independently of any other vertex, each vertex selects one more neighbor uniformly at random from the remaining unchosen neighbors, and an additional directed edge is opened toward that neighbor. In other words, independently at each vertex, the local law corresponds to $k$-DnG with probability $1-\varepsilon$ and to $(k+1)$-DnG with probability $\varepsilon$. In the remainder of this section, we denote the above-defined {\em local $2dp$-DnG law} by $\mathbf{Q}_{p}$ and the corresponding probability measure on the configuration set $\Omega := \{0,1\}^\mathcal{E}$, where $\mathcal{E} := \{(x,y) \in \Z^d\times \Z^d\colon \|x-y\|_{1} = 1\}$ stands for the set of directed edges of $\mathbb{Z}^d$, by $\mathbb{Q}_{p}$. As usual, a directed edge $e$ is called \textit{open} (respectively \textit{closed}) in the configuration $\omega \in \Omega$ if $\omega(e) = 1$ (respectively $\omega(e)=0$). 

We now want to use Theorem~\ref{thm:local-to-global} and a comparison to Bernoulli bond percolation to show that degree-constrained models percolate in certain parameter regimes. For this, we consider convex combinations of $2dp$-DnGs defined as follows.
\begin{definition}(Exchangeability)\label{def_exchange}
    We say that a probability measure $\mathbf P$ on $\mathcal{N}_o$ is {\sl exchangeable} if, for all $k \in \{0,\ldots,2d\}$ with $\mathbf P[ |N(o)|=k ] > 0$, the conditional measure $\mathbf P\big[ \cdot \big| |N(o)|=k\big]$ is a uniform distribution on $\{A\subset\cN_o\colon |A|=k\}$. Let $\mathbf E$ denote the expectation associated with $\mathbf P$.
\end{definition}

Then, the following result states that locally any convex combination of $2dp$-DnGs can be dominated from above and below by other nearest-neighbor local laws with the same expected degree. More precisely, among all probability measures $\mathbb{P}$ for which the marginal distribution $\mathbf{P}$ is exchangeable and the average degree equals $\mathbf{E}\left[ |N(o)| \right]=2dp$, the \emph{all-or-nothing law} $\mathbf P_p^{\mathrm{aon}}$, i.e., the measure for which
    \begin{equation*}
        \mathbf P_p^{\mathrm{aon}}[|N(o)| = 2d] = p\qquad\text{and}\qquad  \mathbf P_p^{\mathrm{aon}}[|N(o)| = 0] = 1-p,\qquad p\in [0,1],
    \end{equation*}
    is the least likely to percolate, whereas the local $2dp$-DnG law, which distributes its mass on $k$ and $k+1$, has the best chance for percolation. 

\begin{proposition}(Domination for exchangeable measures)\label{prop-aonk}
For all exchangeable probability measures $\mathbf P$ on $\mathcal{N}_o$ with $\mathbf E\left[ |N(o)| \right] = 2dp$, one has     \begin{equation}\label{eq:comparison measures}
         \mathbf P_p^{\mathrm{aon}}[N(o) \cap A\neq\emptyset] \leq \mathbf{P}[N(o) \cap A\neq\emptyset]\leq \mathbf{Q}_{p}[N(o) \cap A\neq\emptyset],\quad\text{for all $A \subset \mathcal{N}_o$}.
    \end{equation}
\end{proposition}

\begin{proof}
    We start with the proof of the second inequality in \eqref{eq:comparison measures}.
    Define $\alpha_n = \mathbf P[|N(o)|=n]$ and write $\ell=|A|$. Conditionally on $|N(o)|=n$, we have that
    \begin{align*}
        f(n,\ell) :=  \mathbf P\big[N(o) \cap A\neq\emptyset \big| |N(o)|=n\big] 
        =
        1 - \mathbf P\big[N(o) \cap A = \emptyset \big| |N(o)|=n\big] 
        = 1 - \tfrac{\binom{2d-n}{\ell}}{\binom{2d}{\ell}}.
    \end{align*}
    We claim that the mapping $n\mapsto f(n,\ell)$ is concave on its domain $\{0,\ldots, 2d\}$ for any $\ell\in\{0,\ldots, 2d\}$ which we show at the end of the proof.  In particular, it satisfies the discrete concavity criteria
    \begin{align*}
         \frac{1}{2} f(n-i,\ell) + \frac{1}{2} f(n+i,\ell) &\leq f(n,\ell) \quad \text{and}\\
        \frac{1}{2} f(n,\ell) + \frac{1}{2} f(n+2i+1,\ell) &\leq \frac{1}{2} f(n+i,\ell) + \frac{1}{2} f(n+i+1,\ell).
    \end{align*}
    By the law of total probability, we have that
    \be\label{eq:totprob}
        \mathbf P[  N(o) \cap A\neq\emptyset] = \sum_{n=0}^{2d} \mathbf P\big[ N(o) \cap A\neq\emptyset \big| |N(o)|=n\big] \alpha_n
        = \sum_{n=0}^{2d} f(n,\ell) \alpha_n .
    \ee
    Now, assume that $\mathcal{R}(\mathbf P) \coloneqq \max\{n\colon\alpha_n > 0 \} - \min\{n\colon \alpha_n > 0 \} \geq 2$. We define a new probability  measure $\mathbf P^\prime$ with $\mathcal{R}(\mathbf P^\prime) \leq \mathcal{R}(\mathbf P) - 1$ and $\mathbf E^\prime \left[ |N(o)| \right] = \mathbf E\left[ |N(o)| \right]$ (where $\mathbf E^\prime$ denotes the expectation associated with $\mathbf P^\prime$) such that
    \begin{align*}
        \mathbf P^\prime[ N(o) \cap A\neq\emptyset] \geq \mathbf P[ N(o) \cap A\neq\emptyset],\qquad\text{ for all $A\subset \cN_o$}.
    \end{align*}
   Let $n_+\coloneqq \max\{n\colon\alpha_n > 0 \} $ and $n_- \coloneqq \min\{n\colon \alpha_n > 0 \}$, and define $\Tilde n:= (n_++n_-)/2$. If $n_+ + n_-$ is even, define the probability measure $\mathbf P^\prime$ by
    \begin{align*}
        \mathbf P^{\prime}[ |N(o)| = n] =  \begin{cases}
            \alpha_n - \min\{\alpha_{n_+}, \alpha_{n_-}\} & \text{ if } n = n_+ \text{ or } n= n_-,\\
            \alpha_n + 2 \min\{\alpha_{n_+}, \alpha_{n_-}\} & \text{ if } n = \Tilde{n},\\
            \alpha_n & \text{ else. }
         \end{cases} 
    \end{align*}
    If $n_+ + n_-$ is odd, define  the probability measure $\mathbf P^\prime$ by
    \begin{align*}
        \mathbf P^{\prime}[ |N(o)| = n] =  \begin{cases}
            \alpha_n - \min\{\alpha_{n_+}, \alpha_{n_-}\} & \text{ if } n = n_+ \text{ or } n= n_-,\\
            \alpha_n + \min\{\alpha_{n_+}, \alpha_{n_-}\} & \text{ if } n = \Tilde{n}+\frac{1}{2},\\
            \alpha_n + \min\{\alpha_{n_+}, \alpha_{n_-}\} & \text{ if } n = \Tilde{n}-\frac{1}{2},\\
            \alpha_n & \text{ else. } 
         \end{cases} 
    \end{align*}
    It readily follows that $\mathbf E^\prime \left[ |N(o)| \right] = \mathbf E\left[ |N(o)| \right]$ and that $\cR(\mathbf P^\prime) \leq \cR(\mathbf P)-1$. The fact that
    \begin{align}\label{eq_1}
        \mathbf P^\prime[N(o) \cap A\neq\emptyset] \geq \mathbf P[ N(o) \cap A\neq\emptyset ]
    \end{align}
    follows from~\eqref{eq:totprob} and the concavity of $n \mapsto f(n,\ell)$. Applying this idea inductively, we see that there exists a measure $\mathbf P^\prime$ with $\cR(\mathbf P^\prime)\leq 1$ such that $\mathbf P\left[|N(o)|\right] = \mathbf E^\prime\left[|N(o)|\right]$ and which satisfies~\eqref{eq_1} for all $A\subset \cN_o$. The conditions $\cR(\mathbf P^\prime)\leq 1$ and $\mathbf E\left[|N(o)|\right] = \mathbf E^\prime\left[|N(o)|\right]$ directly imply that $\mathbf P^\prime = \mathbf Q_{p}$, which proves the second inequality of~\eqref{eq:comparison measures}.\\
    
    To conclude the second inequality in \eqref{eq:comparison measures}, it remains to prove the claim that the mapping $n\mapsto f(n,\ell)$ is concave on its domain $\{0,\ldots, 2d\}$ for any $\ell\in\{0,\ldots, 2d\}$. We define $g(n,\ell):=1-f(n,\ell)$.  It is clear that showing that the mapping $n\mapsto f(n,\ell)$ is concave on its domain $\{0,,\ldots, 2d\}$ is equivalent to showing that the mapping $n\mapsto g(n,\ell) = \binom{2d-n}{\ell} / \binom{2d}{\ell}$ is convex on the same domain. In turn, it suffices to show that the mapping $n\mapsto \binom{n}{\ell}$ is convex on $\{0, \ldots, 2d\}$, for any $\ell\in\{0,\ldots, 2d\}$. We extend the mapping $n\mapsto \binom{n}{\ell}$ to $\R_+$ by 
    \be 
    \binom{b}{a}=\frac{\Gamma(b+1)}{\Gamma(a+1)\Gamma(b-a+1)}\qquad \text{for }b\geq a\geq 0, 
    \ee 
    and $\binom{b}{a}=0$ for $0\leq b<a$, where $\Gamma$ is the gamma function, and show that $x\mapsto \binom{x}{\ell}$ is convex on $[0,2d]$ for any $\ell\in\{0,\ldots, 2d\}$. The case $\ell=0$ is clear, so suppose $\ell\in[2d]$. By using that $\frac{\mathrm d}{\mathrm d x}\Gamma(x)=\Gamma(x)\psi^{(0)}(x)$, where $\psi^{(0)}$ is the polygamma function, and with $\psi^{(k)}(x)=\frac{\mathrm d^k}{\mathrm d x^k}\psi^{(0)}(x)$, we obtain 
    \be 
    \frac{\mathrm d^2}{\mathrm d x^2}\binom{x}{\ell}=\binom{x}{\ell}\big[(\psi^{(0)}(x+1)-\psi^{(0)}(x-\ell+1))^2+\psi^{(1)}(x+1)-\psi^{(1)}(x-\ell+1)\big].
    \ee 
    Using the recurrence relation 
    \be 
    \psi^{(k)}(x+1)=\psi^{(k)}(x)+\frac{(-1)^kk!}{x^{k+1}}, 
    \ee 
    we arrive at 
    \be 
    \frac{\mathrm d^2}{\mathrm d x^2}\binom{x}{\ell}=\binom{x}{\ell}\bigg[\bigg(\sum_{k=0}^{\ell-1} \frac{1}{x-k}\bigg)^2-\sum_{k=0}^{\ell-1} \frac{1}{(x-k)^2}\bigg]\geq 0, 
    \ee 
    so that we arrive at the desired result.\\
    
    To prove the first inequality of~\eqref{eq:comparison measures}, let $A \subset \cN_o$ be non-empty and let $v\in A$. Then, for every exchangeable measure $\mathbf P$ on $\{0,1\}^{\cN_o}$ with $\mathbf E\left[ |N(o)| \right] = 2dp$ one has
    \begin{align*}
        \mathbf P[ N(o) \cap A\neq\emptyset] &\geq \mathbf P[ v \in N(o)]
        = \sum_{n=0}^{2d} \mathbf P[ v\in N(o) \big| |N(o)|=n] \mathbf P[ |N(o)|=n]\\
        &
         = \sum_{n=0}^{2d} \frac{n}{2d} \mathbf P[ |N(o)|=n] = \frac{1}{2d} \mathbf E\left[ |N(o)| \right] = p = \mathbf P_p^{\mathrm{aon}}[ N(o) \cap A\neq\emptyset]. \qedhere
    \end{align*}
\end{proof}
The above result now implies the anticipated domination of degree-constrained models from below by independent Bernoulli bond percolation, where we recall the local i.i.d.\ law $\mathbf P_p$ from the beginning of Section~\ref{sec_ex} and the local $2dp$-DnG law $\mathbf Q_p$ from the start of Section~\ref{sec:k-nn-calculation}.

\begin{corollary}(Degree constraints help percolation)\label{prop:comparison-k-dng-bernoulli}
    For all $p \in [0,1]$, 
    \begin{equation*}
        \mathbf{P}_p[N(o) \cap A\neq\emptyset]\leq \mathbf{Q}_{p}[N(o) \cap A\neq\emptyset],\qquad\text{for all $A \subset \mathcal{N}_o$}.
    \end{equation*}
\end{corollary}
\begin{proof}
    This is a direct application of the second inequality in \eqref{eq:comparison measures} with $\mathbf P=\mathbf P_p$.
\end{proof}

\begin{remark}
    In the integer case, i.e.\ when $2dp\in\N_0$, the inequality in Corollary~\ref{prop:comparison-k-dng-bernoulli} can also be proved directly by using negative association of the set of edges $N(o)$. 
\end{remark}

Note that in the above statement, the model $\mathbf{P}_p$ refers to a directed version of independent Bernoulli bond percolation. However, as we establish next, the probabilities of one-arm events for directed and undirected Bernoulli bond percolation coincide. To highlight the distinction between directed and undirected connectivity, in the following, we write $\{ o \leftrightsquigarrow \partial B_n \}$ (resp.\ $\{ o \leftrightsquigarrow \infty \}$) for the event that there is a vertex in $\partial B_n$ (resp.\ for all $k \geq 1$ there exists a vertex in $\partial B_k$) that can be reached via an {\em undirected} path from $o$ consisting of open edges. Let $\mathbb P'_p$ denote the law of independent Bernoulli bond percolation restricted to the undirected nearest-neighbor edges of $\Z^d$ with parameter $p\in [0,1]$. The following identity is proved in the proof of~\cite[Lemma~2.5]{CHJK24}.

\begin{lemma}(Directed and undirected i.i.d.\ percolation)\label{lem-dirundir}
For all $p \in (0,1)$ and $n \geq 0$, 
we have 
\[ \P_p[o \rightsquigarrow \partial B_n]= \mathbb P'_p[o \leftrightsquigarrow \partial B_n]. \]
\end{lemma}

\begin{remark}(All-or-nothing and i.i.d.\ site percolation)\label{lem-aonsite}
Let us note that, with the same technique as presented in the proof above, we can also recover a relation, first observed in~\cite[Lemma~2.6]{CHJK24}, between the all-or-nothing probability measure $\mathbb{P}_p^{\mathrm{aon}}$ and i.i.d.\ site percolation, which we denote by $\mathbb P''_p$. With a slight abuse of notation, under $\mathbb P''_p$ the event $\{ o \leftrightsquigarrow \partial B_n \}$ denotes the event that there is a vertex in $\partial B_n$ that can be reached via a path from $o$ consisting of open sites.
For $p \in (0,1)$ and $n \geq 0$, we have 
\[ \P_p^{\mathrm{aon}}[o \rightsquigarrow \partial B_{n+1}]= \mathbb P''_p[o \leftrightsquigarrow \partial B_n]. \]
Note that, together with Lemma~\ref{lem-dirundir} and the first inequality in~\eqref{eq:comparison measures}, we recover the well-known inequality
\begin{equation*}
   \mathbb P''_p[o \leftrightsquigarrow \infty ] = \P_p^{\mathrm{aon}}[o \rightsquigarrow \infty] \leq \P_p[o \rightsquigarrow \infty] = \mathbb P'_p[o \leftrightsquigarrow \infty] .
\end{equation*}
\end{remark}

Next, let $p_c(d)$ denote the percolation threshold for independent Bernoulli (undirected) bond percolation on $\Z^d$ and let $p_c^{\rm D}(d)$ be the critical value of $p$ for percolation in the $2dp$-DnG in $d$ dimensions. Corollary~\ref{prop:comparison-k-dng-bernoulli} and~\ref{lem-dirundir} together with Theorem~\ref{thm:local-to-global} and Kesten's result ~\cite{Kesten1988pc} stating that 
\begin{equation}\label{Kesten} p_c(d) \sim 1/(2d) \qquad \text{ as } d\to\infty \end{equation}
implies the following corollary.

\begin{corollary}[Large dimensions]\label{cor-D}
For all $d\ge 1$, $p_c^{\rm D}(d)\le p_c(d)$ and $2d p_c^{\rm D}(d)< 1+\eps$ for all $\eps>0$ and dimensions $d$ sufficiently large.
\end{corollary}
In words, in the DnG, a single outgoing edge plus a second outgoing edge that is present with a small positive probability is enough to guarantee percolation in all sufficiently large dimensions. 

\medskip
Before this paper, the question of percolation in the $2$-DnG, i.e., with precisely two outgoing edges, has been open for $d \geq 3$. In two dimensions, percolation in the $2$-DnG, and even $p_c^{\rm D}(2)<1/2$, was proven in~\cite{CHJK24}. Since $p_c(2) = 1/2$ in dimension two (see~\cite{Kesten1980b} and the references therein), we note that Corollary~\ref{cor-D} provides an easy proof for $p_c^{\rm D}(2)\le 1/2$, but not for percolation of the $2$-DnG (and for $d=1$, it is trivial that the $2$-DnG percolates). On the other hand, the absence of percolation for $2dp \leq 1$ follows from an easy argument (see~\cite[Proposition 2.1]{JKLT}) for any dimension. This way, the assertion of Corollary~\ref{cor-D} is optimal up to asymptotic equivalence in the limit $d \to \infty$. 
However, even with~\eqref{Kesten} at hand, to the best of our knowledge, $p_c(d) <1/d$ has not been proven for any specific dimension. Thus, while Corollary~\ref{cor-D} implies that the $2$-DnG, i.e., the graph with exactly two outgoing edges, percolates for all sufficiently large $d$, we are currently unable to establish so for any particular value of $d\ge 3$. 

In contrast, it follows from~\cite{YWpercolation} that $p_c(3) \leq 0.347297$. Thus, by Corollary~\ref{cor-D}, the $2dp$-DnG percolates for $p > 0.347297$, i.e., for $2dp>2.083782$. This strengthens the corresponding result in~\cite{CHJK24}, but the integer case $k=2dp=2$ remains unresolved. Moreover, it has been proven in~\cite[Page 14]{gomes2021upper} that $p_c(4) \leq 0.2788$. Again by Corollary~\ref{cor-D}, the $2dp$-DnG percolates for $p > 0.2788$ in four dimensions, i.e., for $2dp> 2.2305 = 8 \cdot 0.2788$. 
Notably, our argument implies percolation of the $2dp$-DnG for $k=2dp=3$ and $d=4$, resolving the last case where the question of percolation in the $k$-DnG has been open for an integer $k \geq 3$. We summarize the state of the art for the integer case, i.e., when $\eps=0$, in Table~\ref{table-BDU} below. Note that the upper bounds of~\cite{gomes2021upper} on $p_c(d)$ in low dimensions imply that $2dp_c^{\rm D}(d) < 3$ also holds for all $5 \leq d \leq 9$ (while in~\cite{JKLT} only percolation for $k=3$ was shown), but percolation for $k=2$ cannot be derived from them in any dimension.

Another $2dp$-neighbor model is the \emph{undirected $2dp$-neighbor graph} $2dp$-UnG, which is the undirected graph where an undirected nearest-neighbor edge $\{ u,v \}$ of $\Z^d$ is open if and only if $u \in N(v)$ or $v \in N(u)$. Since percolation in the $2dp$-DnG implies percolation in the $2dp$-UnG, our results also affirmatively resolve previously open cases for the undirected model---specifically, percolation for $k=3$ in dimension $d=4$ as well as $k=2$ in sufficiently high dimensions. Again, Table~\ref{table-BDU} summarizes the current state of knowledge of percolation in the $2dp$-UnG (and in a third $2dp$-nearest-neighbor model called the $2dp$-BnG, which will be introduced in Section~\ref{sec-B} below) in low dimensions.

\begin{table}[h!]
\centering
\scalebox{0.78}{
\begin{tabular}{|l|c|c|c|c|c|c|c|c|} 
 \hline
\diagbox{$d$}{$k$} & 1 & 2 & 3 & 4 & 5 & 6 & $\geq 7$ \\ \hline 
1 & no/no/no & yes/yes/yes & - & - & - & - & - \\ \hline
2 & no/no/no & no/\color{blue}yes\color{black}/yes & open/yes/yes & yes/yes/yes & - & - & - \\ \hline
3 & no/no/no & no/open/open & no/\color{blue}yes\color{black}/yes & open/yes/yes & yes/yes/yes & yes/yes/yes & - \\ \hline
4 & no/no/no & no/open/open & no/\color{red}yes/yes\color{black} & open/yes/yes & open/yes/yes & \color{red}yes\color{black}/yes/yes & yes/yes/yes \\ \hline
5 & no/no/no & no/open/open & no/yes/yes & open/yes/yes & open/yes/yes & \color{red}yes\color{black}/yes/yes & \color{red}yes\color{black}/yes/yes \\  \hline
$\geq 6$ & no/no/no & no/open/open & no/yes/yes & open/yes/yes & open/yes/yes & open/yes/yes & open/yes/yes \\  \hline
\text{large} & no/no/no & no/\color{red}yes/yes\color{black} & no/yes/yes & no/yes/yes & no/yes/yes & no/yes/yes & open/yes/yes \\ \hline
\end{tabular}
}
\caption{
Percolation results and open cases for the BnG/DnG/UnG in small dimensions and integers $k$. For given $k,d$, ``yes'' means that the given graph percolates, ``no'' means that it does not, while ``open'' means that the given case is (at least partially) open. Black: proven in~\cite{JKLT}, blue: in~\cite{CHJK24}, red: in our paper. The row titled ``$\geq 6$'' describes the case $d \geq 6$ in general, which still includes open cases. The row titled ``large'' collects results that are true for all sufficiently large $d$. Here, the case of the BnG for $k \geq 7$ is still summarized as ``open'' for short, but we know from Proposition~\ref{prop-B} that the critical value of $k$ for percolation is roughly $\sqrt{2d}$ for large $d$.
}\label{table-BDU}
\end{table}

\subsection{Bidirectional nearest-neighbor models}\label{sec-B}
Let $p \in [0,1]$ and recall the local law $\mathbf Q_p$ of the $2dp$-DnG based on the parameter pair $(k,\eps)$ defined in~\eqref{kepsdef}. $\mathbf Q_p$ can be used to define another law $\mathbb Q^{\rm B}_{p}$, on the {\em undirected} nearest-neighbor edges in $\Z^d$, by declaring any such edge $\{ u, v \}\subset \Z^d$ as open if and only if $u \in N(v)$ and $v \in N(u)$. We call this the \emph{bidirectional $2dp$-neighbor graph} $2dp$-BnG. For $p_c^{\rm B}(d)$, the critical value of $p$ for percolation in the $2dp$-BnG in $d$ dimensions, we have the following result.
\begin{proposition}(Large dimensions for bidirectional models)\label{prop-B}
We have that \[ \lim_{d\to\infty} p_c^{\rm B}(d)\sqrt{2d} = 1. \numberthis\label{B-asymp} \]
\end{proposition} 
The proof, which we present further below, is based on the following generalized version of the local-to-global comparison, where we write $\P^{\rm B}$ for the global law associated to the bidirectional graph based on the local law $\mathbf P$. 
\begin{theorem}(Local-to-global for bidirectional models)\label{theorem-oAB}
    If $\mathbf{P}$, $\mathbf{Q}$ are such that for any two disjoint nonempty sets $A,B \subset \mathcal{N}_o$
    \begin{equation}\label{local-comparison-2}
        \mathbf{P}[N(o) \cap A\neq\emptyset \text{ and } N(o) \cap B\neq\emptyset] \leq \mathbf{Q}[N(o) \cap A\neq\emptyset \text{ and } N(o) \cap B\neq\emptyset] ,
    \end{equation}
    then 
    \begin{equation}
        \mathbb{P}^{\rm B}[o \leftrightsquigarrow \partial B_n] \leq \mathbb{Q}^{\rm B}[o \leftrightsquigarrow \partial B_n]. 
    \end{equation}
\end{theorem}

\begin{proof}
    The proof follows a similar outline as the proof of Proposition~\ref{prop:discrete-interpolation}. We note the required changes. As in the proof of Theorem~\ref{thm:local-to-global}, we let $\mathcal F_{a^{\rm c}}$ denote the $\sigma$-algebra generated by the set-valued random variables $N(u)$ for $u\neq a$. Conditionally on $\mathcal F_{a^{\rm c}}$, we can again distinguish between three cases, of which case $(3)$ is the only non-trivial one. In this case, we let $A$ (resp.\ $B$) denote the set of neighbors $u$ (resp.\ $v$) of $a$ which are connected to the boundary of $B_n$ (resp.\ the origin) and that satisfy that $a\in N(u)$ (resp.\ $a\in N(v)$). Applying~\eqref{local-comparison-2} in a similar manner as~\eqref{local-comparison} is applied in~\eqref{eq:discrinterp} then yields the desired result.
\end{proof} 
Recall the local i.i.d.\ law $\mathbf P_\cdot$ from the beginning of Section~\ref{sec_ex} and the local $2dp$-DnG law $\mathbf Q_\cdot$ from the start of Section~\ref{sec:k-nn-calculation}. 
\begin{lemma}(Interpolation for bidirectional models)\label{lemma-bidir}
    Fix $p \in [0,1]$, $d\ge 1$, and $k\in\{0,\dots,2d\}$. If $k\ge 1+2dp$, then for any two disjoint nonempty sets $A,B \subset \mathcal{N}_o$,
    \begin{equation}\label{local-comparison-2-bi}
        \mathbf{P}_p[N(o) \cap A\neq\emptyset \text{ and } N(o) \cap B\neq\emptyset] \leq \mathbf{Q}_{k/(2d)}[N(o) \cap A\neq\emptyset \text{ and } N(o) \cap B\neq\emptyset].
    \end{equation}
\end{lemma}
\begin{proof}
    Suppose that $\cN_o=\{v_1,\ldots,v_{2d}\}$, with $A=\{v_1,\ldots,v_n\}, B= \{v_{n+1},\ldots,v_{n+m}\}$. Define the events $(\mathcal{A}_i)_{1\leq i \leq n}$ by
    \begin{align*}
        & \mathcal A_i := \{v_i\in N(o) \} \cap \bigcap_{\ell=1}^{i-1} \{ v_\ell \notin N(o) \} .
    \end{align*}
    Observe that the events $\left(\mathcal{A}_i\right)_{1 \leq i \leq n}$ are mutually exclusive and that
    \begin{align*}
        \left\{ N(o) \cap A\neq\emptyset \text{ and } N(o) \cap B\neq\emptyset \right\} = \bigcup_{i=1}^{n} ( \mathcal A_i \cap \left\{ N(o) \cap B\neq\emptyset \right\}).
    \end{align*}
    Conditionally on the event $\mathcal A_i$, we have that
    \begin{align*}
        \mathbf Q_{k/(2d)}[ N(o) \cap B\neq\emptyset | \mathcal A_i]& = \mathbf Q_{k/(2d)}[ N(o) \cap B\neq\emptyset | v_{i} \in N(o), v_1, \ldots, v_{i-1} \notin N(o)]
        \\
        &
        =
        \mathbf Q_{(k-1)/(2d)}[ N(o) \cap B\neq\emptyset | v_1, \ldots, v_{i} \notin N(o)]\\
        &
        \geq
        \mathbf Q_{(k-1)/(2d)}[ N(o) \cap B\neq\emptyset ] \\
        &
        \geq \mathbf P_{(k-1)/(2d)}[ N(o) \cap B\neq\emptyset],
    \end{align*}
    where the last inequality follows from Corollary \ref{prop:comparison-k-dng-bernoulli}. 
    By the law of total probability, we have 
    \begin{align*}
         \mathbf Q_{k/(2d)}[ N(o) \cap A\neq\emptyset \text{ and } N(o) \cap B\neq\emptyset ] 
         &= \sum_{i=1}^{n} \mathbf Q_{k/(2d)}[ N(o) \cap B\neq\emptyset | \mathcal A_i ] \mathbf Q_{k/(2d)}[ \mathcal A_i]\\
         &
        \geq
        \sum_{i=1}^{n} \mathbf P_{(k-1)/(2d)}[N(o) \cap B\neq\emptyset]  \mathbf Q_{k/(2d)}[ \mathcal A_i]
        \\
        &
        =
        \mathbf P_{(k-1)/(2d)}[ N(o) \cap B\neq\emptyset]\mathbf Q_{k/(2d)}[ N(o) \cap A\neq\emptyset].
    \end{align*}
    Another application of Corollary \ref{prop:comparison-k-dng-bernoulli} yields that
    \begin{align*}
        \mathbf P_{(k-1)/(2d)} &[ N(o) \cap B\neq\emptyset]\mathbf Q_{k/(2d)}[ N(o) \cap A\neq\emptyset] \\
        & \qquad 
        \geq
        \mathbf P_{(k-1)/(2d)}[ N(o) \cap B\neq\emptyset] \mathbf P_{k/(2d)}[ N(o) \cap A\neq\emptyset]
        \\
        & \qquad 
        \geq
        \mathbf P_{p}[ N(o) \cap B\neq\emptyset]\mathbf P_{p}[ N(o) \cap A\neq\emptyset]\\
        & \qquad 
        =
        \mathbf P_{p}[ N(o) \cap B\neq\emptyset\text{ and } N(o) \cap A\neq\emptyset],
    \end{align*}
    as desired.
\end{proof}

Note that for the measure $\mathbb{P}^{\rm B}_p$, i.e., the bidirectional version of the directed Bernoulli bond percolation based on $\mathbf{P}_p$, one has that for all pairs of nearest neighbors $u,v\in \Z^d$,
\begin{equation*}
    \mathbb{P}^{\rm B}_p\left[ \{u,v\} \text{ open} \right] = \mathbf{P}_p\left[ u \in N(v) \right] \mathbf{P}_p\left[ v \in N(u) \right] = p^2
\end{equation*}
and that the states of different edges are independent. In particular, in dimensions $d \geq 2$, we have that $\inf_{n \in \N}\mathbb{P}^{\rm B}_p\left[ v \leftrightsquigarrow \partial B_n \right] > 0$ for $p > \sqrt{p_c(d)}$. Using Theorem \ref{theorem-oAB} and Lemma \ref{lemma-bidir}, we get that
\begin{equation*}
    \mathbb{Q}^{\rm B}_{k/(2d)} \left[ 0 \leftrightsquigarrow \infty \right] = \inf_{n \in \N} \mathbb{Q}^{\rm B}_{k/(2d)} \left[ 0 \leftrightsquigarrow \partial B_n \right]
    \ge
    \inf_{n \in \N} \mathbb{P}^{\rm B}_p \left[ 0 \leftrightsquigarrow \partial B_n \right]
    =
    \mathbb{P}^{\rm B}_p \left[ 0 \leftrightsquigarrow \infty \right] > 0
\end{equation*}
if $(k-1)/(2d) \geq p > \sqrt{p_c(d)}$. Using the fact that $p_c(d) \sim 1/(2d)$, we obtain the following corollary.

\begin{corollary}\label{cor-B}
For any $d \geq 2$, if $2dp > 1+2d \sqrt{p_c(d)}$ with $2dp \in \N$, then the $2dp$-BnG percolates. In particular, for any $\delta>0$, if $k \geq  (1+\delta)\sqrt{2d}$, then the $2dp$-BnG percolates almost surely for all sufficiently large dimensions $d$.
\end{corollary}

\begin{proof}[Proof of Proposition~\ref{prop-B}]
For the proof, let $c(d)$ denote the \textit{connective constant} of $\Z^d$, see e.g.~\cite{Grimmett1999}, i.e, 
\begin{align*}
    c(d) := \lim_{n \to \infty}c_n(d)^{1/n},
\end{align*}
where $c_n(d)$ is the number of self-avoiding paths of length $n$ from the origin in the $d$-dimensional hypercubic lattice. In \cite{JKLT} it was shown that if $\eps=0$ and $k(k-1) < 2d(2d-1)/c(d)$, then the associated $2dp$-BnG does not percolate. Since $c(d) \leq 2d-1$, we see that if $k^2 \leq 2d$, then $k(k-1) < 2d(2d-1)/c(d)$ holds and thus the associated $2dp$-BnG does not percolate. In particular, this shows that the $2dp$-BnG does not percolate for $2dp \leq \sqrt{2d}$ when $2dp \in \N$, or equivalently, that the $2dp$-BnG does not percolate for $2dp \leq \lfloor\sqrt{2d}\rfloor$. Dividing by $2d$ on both sides of the inequality, we get that $p_c^B(d) \geq \lfloor\sqrt{2d}\rfloor/2d$ and thus
\be\label{eq:Bliminf}
\liminf_{d\to\infty}p_c^B(d)\sqrt{2d}\geq 1.
\ee 
The other direction $\limsup_{d\to\infty}p_c^B(d)\sqrt{2d}\leq 1$ 
is a direct consequence of Corollary~\ref{cor-B}.
\end{proof}

Let us mention some low-dimensional, previously open cases where we now obtain percolation in the $2dp$-BnG. Corollary \ref{cor-B} guarantees that for $2dp > 8\sqrt{p_c(d)}+1$ with $2dp \in \N$ the $2dp$-BnG percolates.
For $d=4$, the upper bound $p_c(4) \leq 0.2788$ (see \cite[Page 14]{gomes2021upper}) thus guarantees that for $2dp \geq 6$ the $2dp$-BnG percolates. In dimension $d=5$ one has $p_c(5) \leq 0.2284$ (see \cite[Page 14]{gomes2021upper}). This shows that the $2dp$-BnG percolates in dimension $d=5$ when $2dp \geq 6$. For integer $k=2dp$, this result is new when $k \in \{6,7\}$.
See Table~\ref{table-BDU} for a summary of results and open cases of percolation in the $2dp$-BnG, $2dp$-DnG and $2dp$-UnG, for $2dp \in  \N$. 

\subsection{Directed $k$-nearest-neighbor models with additional geometry constraints}
Let us briefly also investigate the following generalized setup of directed $2$-nearest-neighbor percolation on $\Z^2$, first discussed in~\cite{CHJK24}. Instead of opening all possible pairs of outgoing edges with the same probability $1/6$, we distinguish between {\em corners}, i.e., combinations north-west, north-east, south-west, south-east, and {\em sticks}, i.e., combinations north-south and east-west. In order to keep isotropy in the distribution we say that corners are opened with probability $\alpha$ and sticks are opened with probability $\beta$, obeying $4\alpha+2\beta=1$. Hence we have a one-parameter family of measures $(\mathbf P_\alpha)_{0\le\alpha\le 1/4}$ giving rise to measures $(\mathbb P_\alpha)_{0\le\alpha\le 1/4}$ on $\Omega$ that govern our percolation system. Note that, for $\alpha\neq 1/6$ the measures are not exchangeable in the sense of Definition~\ref{def_exchange}. Let $\theta(\alpha)=\P_\alpha[o\rightsquigarrow\infty]$ denote the associated percolation function and note that $\theta(0)=1$. It was shown in~\cite{CHJK24} that the pure corner model percolates, i.e., $\theta(1/4) > 0$, and we have additional numerical evidence that $\alpha\mapsto\theta(\alpha)$ is strictly decreasing and even concave. 

Upper bounding with an i.i.d.\ model does not work since the event of having at least one in  three outgoing edges has probability one under our model but not under the i.i.d.\ model. However, lower bounding with critical i.i.d.\ Bernoulli bond percolation gives a non-trivial percolation regime. More precisely, let $\mathcal{A}_{s}$ denote the event of having an open edge towards the south. Similar let $\mathcal{A}_{se}$, $\mathcal{A}_{sn}$, and $\mathcal{A}_{sen}$, denote the events associated to having at least one open edge towards south/east, south/north, and south/east/north, respectively. Then, 
\begin{align*}
\mathbf P_\alpha[\mathcal{A}_s]&=1/2,\ 
\mathbf P_\alpha[\mathcal{A}_{se}]=3\alpha+2\beta=1-\alpha,\
\mathbf P_\alpha[\mathcal{A}_{sn}]=4\alpha+\beta= 1/2+2\alpha, \
\mathbf P_\alpha[\mathcal{A}_{sen}]=1,
\end{align*}
and, in particular, for $1/8\le \alpha\le 1/4$ all the above probabilities are $\ge$ the associated probabilities for critical i.i.d.\ Bernoulli bond percolation. Hence, any additional small probability to see an additional open edge, together with our comparison result, makes the model with $1/8\le \alpha\le 1/4$ supercritical. 

Here are two more related models for which our method fails. Consider the following directed percolation model in $\Z^2$. Pick one outgoing edge uniformly at random, then flip a coin with success probability $\eps\in [0,1]$ and in case of success, open the edge in the opposite direction as the initially chosen edge. The resulting model is isotropic and has expected degree $1+\eps$. Note that for $\eps=1$, we have $\theta(1)=1$ and $\theta(0)=0$. What can we say about the critical $\eps_c$? In the same vein, consider the model where one outgoing edge is picked uniformly at random and, in case of a successful $\eps$-coin flip, exactly one of the two edges that is perpendicular to the already chosen one is opened uniformly at random. This model is a soft version of the corner percolation model, whereas the model first described is a soft version of the stick percolation model. 

\subsection*{Acknowledgments}
BJ and JK received support by the Leibniz Association within the Leibniz Junior Research Group on \textit{Probabilistic Methods for Dynamic Communication Networks} as part of the Leibniz Competition (grant no.\ J105/2020).
BJ received support from Deutsche Forschungsgemeinschaft through DFG Project no.\ P27 within the SPP 2265. BL has received funding from the European Union’s Horizon 2022 research and innovation programme under the Marie Sk\l{}odowska-Curie grant agreement no.\ $101108569$. LR was partially supported by NSF grant DMS-2303316.
Funding acknowledgements by AT: This paper was supported by the János Bolyai Research Scholarship of the Hungarian Academy of Sciences.  Project no.\ STARTING 149835 has been implemented with the support provided by the Ministry of Culture and Innovation of Hungary from the National Research, Development and Innovation Fund, financed under the STARTING\_24 funding scheme.

\bibliographystyle{alpha}
\bibliography{references}

\end{document}